\documentclass{amsart}
\usepackage{amsmath,amsfonts,amssymb}
\usepackage{cite}
\usepackage{mathrsfs}
\usepackage{dsfont}

\newtheorem{theorem}{Theorem}[section]

\newtheorem{corollary}[theorem]{Corollary}

\newtheorem{definition}[theorem]{Definition}

\newtheorem{remark}[theorem]{Remark}

\newcounter{figures}[section]

\def\bN{{\mathbb N}}

\def\bR{{\mathbb R}}

\def\cC{\mathcal{C}}

\def\cF{\mathcal{F}}

\def\cM{\mathcal{M}}

\def\cS{\mathcal{S}}

\def\supp{{\rm{supp}\, }}

\def\cprime{$'$}

\newcommand{\brac}[1]{\langle #1\rangle_{{\vec{a}}}}

\begin{document}
\title[Fourier multipliers on anisotropic mixed-norm spaces]
{Fourier multipliers on anisotropic mixed-norm spaces of distributions}

\author{Galatia Cleanthous}
\address{Department of Mathematics\\
Aristotle University of Thessaloniki\\ Thessaloniki 541 24, Greece}
\email{galatia.cleanthous@gmail.com}

\author{Athanasios G. Georgiadis}

\author{Morten Nielsen}
\address{Department of Mathematical Sciences\\ Aalborg University\\ Skjernvej 4A\\ DK-9220 Aalborg\\ Denmark}
\email{nasos@math.aau.dk}
\email{mnielsen@math.aau.dk}

\begin{abstract}
A new general H\"{o}rmander type condition involving anisotropies and mixed norms is introduced, and boundedness results for Fourier multipliers on anisotropic Besov and Triebel-Lizorkin spaces of distributions with mixed Lebesgue norms are obtained. As an application, the continuity of such operators is established on mixed Sobolev and Lebesgue spaces too. Some lifting properties and equivalent norms are obtained as well.
\end{abstract}

\maketitle

\section{Introduction}\label{Introduction}
\setcounter{equation}{0}
The study of spaces of functions and distributions and operators on such spaces play an essential role in harmonic analysis. Several  branches of both pure and applied mathematics make extensive use of such spaces,  including in the study of partial differential equations, approximation theory, probability, statistics, and signal processing.

Some of the most general and applicable families of functions spaces in analysis are the Besov and Triebel-Lizorkin spaces. The two families are interesting in their own right, but their importance also stem from the fact that several of the classical function spaces such as Lebesgue, Hardy, BMO, Sobolev, and H\"{o}lder spaces can be recovered as special cases. The  Besov and Triebel-Lizorkin spaces have been studied for many different reasons and in a variety of settings and circumstances. For further details we refer the reader to \cite{B0,BH,CGN,JS,KP,KPY,LYY,YY} and to the references found therein. 

The purpose of this article is to study Fourier multipliers in the general setting of anisotropic Triebel-Lizorkin spaces based on mixed-norm Lebesgue spaces. Let us now elaborate further on this particular setting.

Anisotropic phenomena appear naturally in various fields of analysis, both pure and applied. A classical example found in \cite{KG} is the case of differential operators with anisotropic symbols. Such operators naturally introduce a need for anisotropic function classes containing functions compatible with the particular anisotropy. The notion of anisotropic Besov spaces goes back to Nikol{\cprime}ski{\u\i} \cite{Nik} and the notion of anisotropic Triebel-Lizorkin spaces can be found in Triebel's book \cite[p.\ 269]{T}. More recently, anisotropic Besov spaces in a more general setting have been studied by Bownik \cite{B0} and the anisotropic Triebel-Lizorkin spaces have been studied by Bownik and Ho \cite{BH}.

Mixed-norm Lebesgue spaces are useful as a framework for several  problems arising in physics demanding different regularity in every direction (e.g.\ in time and in space). Mixed-norm Besov and Triebel-Lizorkin spaces have been studied during this decade by Johnsen and Sickel as one may see for example \cite{JS}. Here we will work on anisotropic mixed-norm Besov and Triebel-Lizorkin spaces.

Fourier multipliers form one of the fundamental and most important classes  of operators  in harmonic analysis. Their importance is emphasized by their close link to partial differential operators through the Fourier transform, and there has been a continuous interest in the study of boundedness properties of multipliers on $L^p$ and other spaces  since the seminal work by  Marcinkiewicz \cite{Mar}, Mihlin \cite{M} and H\"{o}rmander \cite{H}. Numerous variations and generalizations of the above works have been produced during the past years, see \cite{BB,BB2,CK,DOS,DY,G,KM,LM,YYZ,Hy2} and the references therein.

In this article we prove boundedness of a suitable class of Fourier multipliers on anisotropic Besov and Triebel-Lizorkin spaces with mixed Lebesgue norms. Moreover, we also  introduce a new general H\"{o}rmander-type class of multipliers naturally adapted to mixed norms and  to the general anisotropic setting.

Let us summarize the main contributions in this paper.

$(\alpha)$ We introduce a new and general condition, involving mixed norms and the anisotropy, for H\"{o}rmander multipliers, see (\ref{A_tp}).

$(\beta)$ We prove the boundedness of Fourier multipliers on anisotropic mixed-norm Besov and Triebel-Lizorkin space, see Theorem \ref{Thmain}.

$(\gamma)$ The continuity of Fourier multipliers on mixed Lebesgue and Sobolev spaces will be obtained, under the new condition as well, see Corollaries \ref{Csobolev}-\ref{Csobolev3}.

$(\delta)$ An equivalent norm characterization for the anisotropic mixed-norm Besov and Triebel-Lizorkin spaces is revisited, see Corollary \ref{Cequivallent}.

\subsection*{Notation} We will denote by $\cF(f)(\xi):=\hat{f}(\xi):=\int_{\bR^n}f(x)e^{-ix\cdot\xi}dx$  the Fourier  transform of (suitably nice) $f$, where $x\cdot\xi:=x_1\xi_1+\cdots+x_n\xi_n$  is the standard inner product on $\bR^n$. The inverse Fourier transform is then given by $\cF^{-1}f(x):=\hat{f}(-x)$.
For $\vec{t}=(t_1,\dots,t_n)\in \bR^n$ with $t_1,\dots,t_n\neq0$, we set $\frac{1}{\vec{t}}:=\big(\frac{1}{t_1},\dots,\frac{1}{t_n}\big)$. The sets of positive and non-negative integers will be denoted by $\bN$ and $\bN_0$ respectively. For  $\gamma$  a multi-index $\gamma=(\gamma_1,\dots,\gamma_n)\in\bN_0^n$, we denote by $|\gamma|:=\gamma_1+\cdots+\gamma_n$ its length and we set $\partial^\gamma f:=\partial^{\gamma_1}_{x_1}\cdots\partial^{\gamma_n}_{x_n}f$. By $\cS:=\cS(\mathbb{R}^n)$ we denote the Schwartz class on $\mathbb{R}^n$ and by $\cS'$ its dual; the tempered distributions. The $N\in\bN_0$ times differential functions on $\mathbb{R}^n$ is denoted by $\mathcal{C}^N$. Finally, any positive constant will be denoted $c$, or as $c_{\alpha}$ if it depends on a significant parameter $\alpha$. 

\section{Preliminaries}\label{preliminaries}
\subsection{Mixed norm Lebesgue spaces}
Let $\vec{p}=(p_1,\dots,p_n),$ with $0<p_1,\dots,p_n\le\infty$ and let $f:\mathbb{R}^n\rightarrow \mathbb{C}$ be measurable. We say that $f\in L^{\vec{p}}=L^{\vec{p}}(\mathbb{R}^n)$ if
\begin{equation*}
\|f\|_{\vec{p}}:= \left(\int_\mathbb{R}\cdots\left(\int_\mathbb{R}\left(\int_\mathbb{R} |f(x_1,\dots,x_n)|^{p_1} dx_1\right)^{\frac{p_2}{p_1}} dx_2\right)^{\frac{p_3}{p_2}}\cdots dx_n\right)^{\frac{1}{p_n}}<\infty,
\end{equation*}
with the standard modification when $p_k=\infty,$ for some $1\le k\le n.$ The quasi-norm $\|\cdot\|_{\vec{p}},$ is a norm when $\min(p_1,\dots,p_n)\ge1$ and turns $L^{\vec{p}}$ into a Banach space. For further properties of $L^{\vec{p}}$ see for example \cite{AI,Ba,BP,F,L}.

Let $R\subset\bR^n$. We denote by $\|f\|_{L^{\vec{p}}(R)}:=\|f\chi_R\|_{\vec{p}}$, where $\chi_R$ is the characteristic function of $R$. For example when $R:=I_1\times\cdots\times I_n\subset\bR^n$ is a rectangle, we obtain
\begin{equation*}
\|f\|_{L^{\vec{p}}(R)}= \left(\int_{I_n}\cdots\left(\int_{I_2}\left(\int_{I_1} |f(x_1,\dots,x_n)|^{p_1} dx_1\right)^{\frac{p_2}{p_1}} dx_2\right)^{\frac{p_3}{p_2}}\cdots dx_n\right)^{\frac{1}{p_n}}.
\end{equation*}

For $\vec{p}\in[1,\infty]^n$ we define the \textit{conjugate} $\vec{p'}:=(p_1',\dots,p_n')\in [1,\infty]^n$ by requiring that $1/p_k+1/p_k'=1$ for every $k=1,\dots,n$.\

The \textit{mixed H\"{o}lder-inequality} (see e.g. \cite{BP}) is the following estimate: For every $\vec{p}=[1,\infty]^n$, $f\in L^{\vec{p}}$ and $g\in L^{\vec{p'}}$ we have
\begin{equation}\label{mixedholder}
\left|\int_{R^n} f(x)\overline{g}(x) dx\right|\leq \|f\|_{\vec{p}}\|g\|_{\vec{p'}}.
\end{equation}

We will also need an adapted version of the Hausdorff-Young inequality. The \textit{mixed Hausdorff-Young's Theorem} \cite{BP} asserts that if $\vec{t}=(t_1,\dots,t_n)$ with $1\leq t_n\leq t_{n-1}\le \cdots \leq t_1\leq 2$, then for every $f\in\cS$,
\begin{equation}\label{Penedek}
\|\hat{f}\|_{\vec{t'}}\leq \|f\|_{\vec{t}}.
\end{equation}

\subsection{Anisotropic geometry}\label{subsec:AD}
Let $b,x\in\mathbb{R}^n$ and $\lambda>0$. We denote by $\lambda^{b}x:=(\lambda^{b_1}x_1,\dots,\lambda^{b_n}x_n).$ We fix a vector $\vec{a}\in[1,\infty)^n$, and we introduce the anisotropic quasi-homogeneous norm $|\cdot|_{\vec{a}}$ as follows: We set $|0|_{\vec{a}}:=0$, and for $x\neq0$ we let $|x|_{\vec{a}}:=\lambda_0,$ where $\lambda_0$ is the unique positive number such that $|\lambda_0^{-\vec{a}}x|=1.$ One observes immediately that
\begin{equation}\label{ad1}
|\lambda^{\vec{a}}x|_{\vec{a}}=\lambda|x|_{\vec{a}},\;\;\text{for every}\;\;x\in\mathbb{R}^n,\;\lambda>0.
\end{equation}
From this we notice that $|\cdot|_{\vec{a}}$ is not a norm unless $\vec{a}=(1,\dots,1)$, where
it coincides with the 
Euclidean norm.

We have a link between the anisotropic and the Euclidean geometry (see \cite{B,BH}): There exist constants $c_1,c_2>0$ such that for every $x\in\mathbb{R}^n,$ 
\begin{equation}\label{ad8}
c_1(1+|x|_{\vec{a}})^{a_m}\le1+|x|\le c_2(1+|x|_{\vec{a}})^{a_M}.
\end{equation}
where we denoted $a_m:=\min_{1\le j\le n}a_j,\;a_M:=\max_{1\le j\le n}a_j.$

Finally, we will need the so-called \textit{homogeneous dimension}:
\begin{equation}\label{nu}
  \nu:=|\vec{a}|=a_1+\cdots+a_n.
  \end{equation}

\subsection{Anisotropic mixed-norm Triebel-Lizorkin and Besov spaces}\label{subsec:TL}
In this section define the anisotropic mixed-norm smoothness spaces needed for our analysis, and we discuss some corresponding Fefferman-Stein vector-valued maximal function estimates.

Let $\varphi_0\in \mathcal{S}$ (the class of Schwartz functions) be such that
\begin{equation}\label{phi01}
\supp(\widehat{\varphi_0})\subseteq 2^{\vec{a}}[-2,2]^n=:R_0,
\end{equation}
\begin{equation}\label{phi02}
|\widehat{\varphi_0}(\xi)|\geq c>0 \ \ \ \text{if} \ \ \xi\in2^{\vec{a}}[-5/3,5/3]^n.
\end{equation}
and let $\varphi\in \mathcal{S}$ be such that
\begin{equation}\label{phi1}
\supp(\hat{\varphi})\subseteq [-2,2]^n\setminus (-1/2,1/2)^n=:\widetilde{R_1},
\end{equation}
\begin{equation}\label{phi2}
|\hat{\varphi}(\xi)|\geq c>0 \ \ \ \text{if} \ \ \xi\in[-5/3,5/3]^n\setminus (-3/5,3/5)^n.
\end{equation}
Note that it is possible to choose $\varphi,\varphi_0$ satisfying the partition of unity condition
\begin{equation}\label{phi3}
\widehat{\varphi_{0}}(\xi)+\sum\limits_{j=1}^{\infty}\hat{\varphi}(2^{-j\vec{a}}\xi)=1 \ \ \ \text{for every} \ \ \xi\in\mathbb{R}^n.
\end{equation}

We define the ``rectangular version" of the annulus by 
\begin{equation}\label{Rj}
R_j:= 2^{j\vec{a}} \widetilde{R_1},\;j\geq 1\;\;\text{and}\;\;R_j:=\emptyset,\;j<0.
\end{equation}

Note that the punctured rectangle $\widetilde{R_1}$ can be expressed as the (almost) disjoint union of $k_n:=2^{3n}-2^{n}$ closed dyadic cubes $\{Q_\mu\}_{\mu=1,\dots,k_n}$ of side-length $2^{-2}$. Then for every $j\ge1$ we have that $R_j=\cup_{\mu=1}^{k_n}2^{j\vec{a}}Q_\mu$. For every $\vec{p}\in(0,\infty]^n$, we have a two sided estimate of the mixed-norm on $R_j$  as
\begin{equation}\label{mixedRj}
\|f\|_{L^{\vec{p}}(R_j)}\asymp\sum\limits_{\mu=1}^{k_n}\|f\|_{L^{\vec{p}}(2^{j\vec{a}}Q_\mu)},
\end{equation}
where the constants in the equivalence depends only on $\vec{p}$ and $n$.

We denote by $\varphi_{j}(x):=2^{\nu j}\varphi(2^{j\vec{a}}x), \ j\in\mathbb{N}$. We then have $\widehat{\varphi_{j}}(\xi)=\widehat{\varphi}(2^{-j\vec{a}}\xi),$ so by (\ref{phi1})
\begin{equation}\label{phi7}
\supp(\widehat{\varphi_j})\subseteq 2^{j\vec{a}} \supp(\widehat{\varphi})\subseteq R_{j} \ \ \ \text{for every} \ \ j\in\bN.
\end{equation}

Let us now recall the definition of anisotropic mixed-norm Triebel-Lizorkin and Besov spaces (see for example \cite{JS}).\\

For $s\in\mathbb{R},\;\vec{p}\in(0,\infty)^n,\;q\in(0,\infty]$ and $\vec{a}\in[1,\infty)^n,$ the anisotropic mixed-norm Triebel-Lizorkin space $F^s_{\vec{p}q}(\vec{a})$ is defined, as the set of all $f\in \mathcal{S}'$ such that 
\begin{equation}\label{TLnorm}
\|f\|_{F^s_{\vec{p}q}(\vec{a})}:=\Big \| \Big( \sum _{j=0}^{\infty} (2^{sj}|\varphi_{j}\ast f|)^q \Big) ^{1/q}\Big \|_{\vec{p}}<\infty,
\end{equation}
For $s\in\mathbb{R},\;\vec{p}\in(0,\infty]^n,\;q\in(0,\infty]$ and $\vec{a}\in[1,\infty)^n,$ the anisotropic mixed-norm Besov space $B^s_{\vec{p}q}(\vec{a})$ is defined, as the set of all $f\in \mathcal{S}'$ such that 
\begin{equation}\label{Bnorm}
\|f\|_{B^s_{\vec{p}q}(\vec{a})}:= \Big( \sum _{j=0}^{\infty} (2^{sj}\|\varphi_{j}\ast f\|_{\vec{p}})^q \Big) ^{1/q}<\infty,
\end{equation}
with the $\ell_q$-norm replaced by $\sup _{j}$ if $q=\infty$ for both $F^{s}_{\vec{p}q}(\vec{a})$ and $B^{s}_{\vec{p}q}(\vec{a})$. Note that $F^{s}_{pq}(\vec{a})=F^{s}_{\vec{p}q}(\vec{a})$ for $\vec{p}=(p,\dots,p)$ (and the same holds true for the Besov spaces). Further properties of $F^{s}_{\vec{p}q}(\vec{a})$ and $B^{s}_{\vec{p}q}(\vec{a})$ can be found in \cite{JHS,GN,Sch,ST}.

\begin{remark} 
One can easily verify that the definition of mixed-norm Triebel-Lizorkin and Besov spaces based on test functions with Fourier transforms having supports in a classical annulus, such as considered in e.g.\ \cite{JS}, is equivalent to the above definitions.
\end{remark}

\subsection{Maximal operators}\label{subsec:maxop}
Maximal operators will be an essential tool  in the proof of our main result. Let $1\leq k\leq n$. We define for $f\in L^1_\text{loc}(\mathbb{R}^n)$,
\begin{equation}\label{MK}
M_k f(x):=\sup\limits_{I\in I_x^k} \dfrac{1}{|I|} \int_I |f(x_1,\dots,y_k,\dots,x_n)| dy_k,
\end{equation}
where $I_x^k$ is the set of all intervals $I$ in $\mathbb{R}_{x_k}$ containing $x_k$.\

We will use extensively the following \textit{iterated maximal operator}: for $f\in L^1_\text{loc}(\mathbb{R}^n)$ we let
\begin{equation}\label{Max1}
\cM_{\vec{r}} f(x):=\left(M_n\Big(\cdots M_2(M_1|f|^{r_1})^{r_2/r_1}\cdots\Big)^{r_n/r_{n-1}}\right)^{1/r_n}(x),\;\vec{r}\in(0,\infty)^n,\;x\in\mathbb{R}^n.
\end{equation}

We shall need a variation of \textit{Fefferman-Stein} vector-valued maximal inequality (see \cite{Ba,JS}):\
If $\vec{p}=(p_1,\dots,p_n)\in(0,\infty)^n,\;q\in(0,\infty]$ and $\vec{r}=(r_1,\dots,r_n)\in(0,\infty)^n$ with $r_k<\min(p_1,\dots,p_k,q)$ for every $k=1,\dots,n$ then
\begin{equation}\label{max}
\Big\|\Big(\sum_{j\geq0} \big(\cM_{\vec{r}}(f_j)\big)^q
\Big)^{1/q}\Big\|_{\vec{p}}\leq c\Big\|\Big(\sum_{j\ge0} |f_{j}|^q
\Big)^{1/q}\Big\|_{\vec{p}}.
\end{equation}

By \cite[Proposition 3.11]{JS}, we obtain the following compound result: For every $\vec{r}\in(0,\infty)^n$ there exists a constant $c>0,$ such that for every $\vec{b}=(b_1,\dots,b_n)\in(0,\infty)^n$ and $f$ with $\supp(\hat{f})\subset [-b_1,b_1]\times\cdots\times[-b_n,b_n],$
\begin{equation}\label{M3}
\sup_{z\in\mathbb{R}^n} \dfrac{|f(x-z)|}{(1+|b_1z_1|)^{1/{r_1}}\cdots(1+|b_nz_n|)^{1/{r_n}}}\leq c\cM_{\vec{r}} f(x),\;x\in\mathbb{R}^n.
\end{equation}

\section{Fourier multipliers}\label{sec:psiDO}

\subsection{Anisotropic Fourier multipliers}\label{se:mult} One of the most classical problems in harmonic analysis is the boundedness of Fourier multipliers between suitable function (or distribution) smoothness spaces. A bounded function $m=m(\xi)$ on $\mathbb{R}^n$ is called a multiplier. 
The corresponding Fourier multiplier operator is given by 
\begin{equation}
\label{dfm}T_{m}f(x):=\int_{\mathbb{R}^{n}}m(\xi)\hat{f}(\xi)e^{ix\cdot\xi}d\xi,\ \text{for every} \ \ x\in\mathbb{R}^n, \ \ f\in\mathcal{S}.
\end{equation}

The question of boundedness of $T_{m}$ is extremely well studied in the Euclidean setting as well as on manifolds, groups, symmetric spaces, and in many other settings. See for example \cite{G,H,KM,LM,M} and the references therein. 

Fourier multipliers on Triebel-Lizorkin spaces have been studied by Triebel in \cite{T2}.
For anisotropic Besov and Triebel-Lizorkin spaces we refer to the articles \cite{BB,BB2} of B\'{e}nyi and Bownik. 

Yang and Yuan in \cite{YY} introduced some general scales; Triebel-Lizorkin-type spaces. For Fourier multipliers on such spaces, see D.\ Yang et.\ al.\ \cite{YYZ}.

Perhaps the most well-known multiplier conditions (with respect to anisotropic geometry) are the following Mihlin and H\"{o}rmander conditions, which we will state as $L^{\infty}$ and $L^{2}$ conditions for compatibility with the new condition that we are going to introduce below.

Let $\alpha\in\bR,\;N\in\bN$ and $m\in\cC^N$, we say that $m$ satisfies the:

1. $L^{\infty}$-condition when 
\begin{equation}
\label{multiinfty}\sup_{|\gamma|\le N}\sup_{\xi\in\bR^n}\left|(1+|\xi|_{\vec{a}})^{-\alpha+\vec{a}\cdot\gamma}\partial^{\gamma}m(\xi)\right|<\infty.
\end{equation}

2. $L^{2}$-condition when
\begin{equation}
\label{multi2}\sup_{|\gamma|\le N}\sup_{j\ge0}\left \{
2^{-j\alpha}2^{j\vec{a}\cdot\gamma}2^{-j\nu/2}\|\partial^\gamma m\|_{L^{2}(R_j)}\right \}<\infty,
\end{equation}
where $R_j$ are as in (\ref{phi01}) and (\ref{Rj}).

\begin{remark}
1. The conditions (\ref{multiinfty}) and (\ref{multi2}) are the anisotropic analogues of the classical inhomogeneous ones, see \cite{BB,BB2,H,M}, with the extra parameter $\alpha\in\bR$, which allows us to interplay between different smoothness levels as in \cite{CK,YYZ}. 

2. Under the isotropic versions of the above conditions, Antoni\'{c} and Ivec in the recent paper \cite{AI}, proved the boundedness of Fourier multipliers on mixed Lebesgue and Sobolev spaces.

3. It is not hard to see that $\|1\|_{L^{2}(R_j)}=c2^{j\nu/2}$, for every $j\in\bN_0$, so the $L^{2}$-condition is 
sharper than the $L^{\infty}$-condition.

4. Multipliers on anisotropic homogeneous mixed-norm spaces are considered by the present authors in \cite{CGN_ACHA}.
\end{remark}

\subsection{A new class of multipliers} Here we introduce a new class of multipliers replacing the conditions (\ref{multiinfty}) and (\ref{multi2}) above by a mixed-norm $L^{\vec{t}}$-condition of H\"{o}rmander type. Before of this we give the following definition

\begin{definition}\label{def1}
A vector $\vec{t}=(t_1,\dots,t_n)\in [1,2]^n$ with $1\le t_n\leq t_{n-1}\leq \cdots \leq t_1\leq 2$ will be called admissible. We also denote by $t:=\frac{1}{t_1}+\cdots+\frac{1}{t_n}.$
\end{definition}

We proceed now to define multiplier classes with respect to both the anisotropy and the mixed-norms.
\begin{definition}\label{defA_tp}
Let $\vec{a}\in[1,\infty)^n$. Given $\alpha\in\bR$, an admissible $\vec{t}$ and $N\in\bN$,
we say that the multiplier $m\in \cC^N(\mathbb{R}^n)$ satisfies the $L^{\vec{t}}$-condition, or that it belongs to the class $\mathcal{A}(\alpha,\vec{t},N)$, if 
\begin{equation}\label{A_tp}
A_{\alpha,\vec{t},N}(m):=\sup\limits_{|\gamma|\leq N}\sup\limits_{j\geq 0} \left \{
2^{-j\alpha}2^{j\vec{a}\cdot\gamma}2^{-j\vec{a}\cdot\frac{1}{\vec{t}}}\big\|\partial ^\gamma m\big\|_{L^{\vec{t}}(R_j)}\right\}<\infty,
\end{equation}
where the $R_j$'s are as in (\ref{phi01}) and (\ref{Rj}).
\end{definition}

We have the followings remarks pertaining to Definition \ref{defA_tp}:

\begin{remark}\label{rema1}
Note that:

\begin{itemize}
	\item [1.] When $\vec{t}=(2,\dots,2)$, the $L^{\vec{t}}$-condition (\ref{A_tp}) coincides with the H\"{o}rmander $L^2$-condition (\ref{multi2}).

\item [2.] We observe that $\|1\|_{L^{\vec{t}}(R_j)}=c2^{j\vec{a}\cdot\frac{1}{\vec{t}}}$, for every $j\in\bN_0$, and $\vec{t}$ admissible. Then from (\ref{ad8}), we conclude that the $L^{\vec{t}}$-condition is 
sharper than the $L^{\infty}$-condition.

\item [3.] For every $\vec{t},\vec{r}\in[1,2]^n$, with $\vec{t}\le\vec{r}$ (i.e., $t_j\le r_j,\;j=1,\dots,n$) we have 
\begin{equation}\label{1,t,2}
2^{-j\nu}\|f\|_{L^{1}(R_j)}\le2^{-j\vec{a}\cdot\frac{1}{\vec{t}}}\|f\|_{L^{\vec{t}}(R_j)}
\le2^{-j\vec{a}\cdot\frac{1}{\vec{r}}}\|f\|_{L^{\vec{r}}(R_j)}\le2^{-j\nu/2}\|f\|_{L^{2}(R_j)}.
\end{equation}
\end{itemize}

\end{remark}

\subsection{The main result}
We now proceed to state our main result. However, we need first to fix some notation.
\begin{definition}\label{def2}
For every vector $\vec{p}=(p_1,\dots,p_n)\in (0,\infty)^n$ and every $q\in(0,\infty]$, we set $\mu_j:=\min(p_1,\dots,p_j,q)$,$\;j=1,\dots,n$ for the case of Triebel-Lizorkin and  $\mu_j:=\min(p_1,\dots,p_j)$,$\;j=1,\dots,n$ for the case of Besov spaces. In every case we denote by $\mu:=\frac{1}{\mu_1}+\cdots+\frac{1}{\mu_n}.$
\end{definition}

Our main multiplier result is the following.

\begin{theorem}\label{Thmain}
Let $\alpha,s\in\mathbb{R},\;\vec{p}=(p_1,\dots,p_n)\in (0,\infty)^n,\;q\in(0,\infty],\;\vec{a}\in [1,\infty)^n,$ an admissible vector $\vec{t}$ and $N\in\bN$ with $N>\mu+t,$ where $t,\mu$ as in Definitions \ref{def1} and \ref{def2}.

If $m$ is a multiplier in the class $\mathcal{A}(\alpha,\vec{t},N)$, then the Fourier multiplier
$T_m$ is bounded from $F^{s+\alpha}_{\vec{p}q}(\vec{a})$ to $F^{s}_{\vec{p}q}(\vec{a})$
and from $B^{s+\alpha}_{\vec{p}q}(\vec{a})$ to $B^{s}_{\vec{p}q}(\vec{a})$.
\end{theorem}

\begin{proof} We shall treat only the case of Triebel-Lizorkin spaces. The Besov space case is similar and we leave it for the reader.

Let $(\varphi_0,\varphi)$ be a couple satisfying (\ref{phi01})-(\ref{phi3}). From (\ref{phi3}) and bearing in mind (\ref{phi7}), we can verify that there exists $M\in\mathbb{N}$ such that, for every $j\ge0$
\begin{equation}\label{thproof1}
\widehat{\varphi_j}=\sum\limits_{k=j-M}^{j+M} \widehat{\varphi_k}\widehat{\varphi_j},
\end{equation}
with the convention that $\widehat{\varphi_k}\equiv0$ if $k<0$.

Let $f\in F^{s+\alpha}_{\vec{p}q}(\vec{a})$. We have for every $\xi\in\mathbb{R}^n$, using   equation (\ref{dfm}),
\begin{eqnarray}\label{thproof2}\nonumber
\widehat{\varphi_j}(\xi)\widehat{T_m f}(\xi)&=&\widehat{\varphi_j}(\xi)m(\xi)\hat{f}(\xi)\\ 
&=&\sum\limits_{k=j-M}^{j+M} m(\xi)\widehat{\varphi_k}(\xi)\widehat{\varphi_j}(\xi)\hat{f}(\xi).
\end{eqnarray}
We set
\begin{equation}\label{mj}
m_{(j)}(\xi):=2^{-j\alpha}m(\xi) \sum\limits_{k=j-M}^{j+M} \widehat{\varphi_k}(\xi),\;\;\xi\in\mathbb{R}^n,
\end{equation}
and
\begin{equation}\label{gj}
g_{(j)}(\xi):=m_{(j)}(2^{j\vec{a}}\xi),\;\;\xi\in\mathbb{R}^n.
\end{equation}

In the light of the above, and by the inverse Fourier transform $\mathcal{F}^{-1}$, (\ref{thproof2}) implies
\begin{equation}\label{thproof3}
\big(\varphi_j\ast T_mf \big)(x)= 2^{j\alpha}\big(\mathcal{F}^{-1}(m_{(j)})\ast (\varphi_j\ast f)\big)(x),\;\;x\in\mathbb{R}^n.
\end{equation}

From the selection of $N$, we can find $\varepsilon>0$ such that $N=\mu+t+2n\varepsilon$. We set $$N_k:=\frac{1}{\mu_k}+\frac{1}{t_k}+2\varepsilon,\;\;\text{for every}\;\;k=1,\dots,n$$and hence $N_1+\cdots+N_n=N$. We also introduce the vector $\vec{r}:=(r_1,\dots,r_n)$ where $r_k:=(1/\mu_k+\varepsilon)^{-1}$ for every $k=1,\dots,n$ and thus $$\frac{1}{r_k}=N_k-\Big(\varepsilon+\frac{1}{t_k}\Big),\;\;\text{for every}\;\;k=1,\dots,n.$$

By (\ref{thproof3}), we obtain
\begin{eqnarray}\label{thproof4}\nonumber
|\big(\varphi_j\ast T_mf \big)(x)|&\le&2^{j\alpha}\int_{\mathbb{R}^n}|\mathcal{F}^{-1}(m_{(j)})(y)||\varphi_j\ast f(x-y)|dy\nonumber \\
&\leq& 2^{j\alpha}\sup\limits_{z\in\mathbb{R}^n} \dfrac{|\varphi_j\ast f(x-z)|}{\prod\limits_{k=1}^n (1+|2^{ja_k}z_k|)^{1/r_k}}\times I,\;\;x\in\mathbb{R}^n,
\end{eqnarray}
where we have put
\begin{equation}\label{I}
I:=\int_{\mathbb{R}^n} |\mathcal{F}^{-1}(m_{(j)})(y)|\prod\limits_{k=1}^n (1+|2^{ja_k}y_k|)^{1/r_k} dy.
\end{equation}
Since $\supp (\widehat{\varphi_j\ast f})\subset 2^{j\vec{a}}R_1\subset 2^{j\vec{a}}[-2,2]^n$, we use the maximal inequality (\ref{M3}) to obtain that 
\begin{equation}\label{thproof5}
\sup\limits_{z\in\mathbb{R}^n} \dfrac{|\varphi_j\ast f(x-z)|}{\prod\limits_{k=1}^n (1+|2^{ja_k}z_k|)^{1/r_k}}
\leq c\mathcal{M}_{\vec{r}} (\varphi_j\ast f)(x),\;\;x\in\mathbb{R}^n.
\end{equation}

\noindent \textit{Estimation of I.}\ \\

By the definition of $g_{(j)}$ we have $\mathcal{F}^{-1}(m_{(j)})(y)=2^{j\nu}\mathcal{F}^{-1}(g_{(j)})(2^{j\vec{a}}y)$, so
\begin{eqnarray}\label{thproof6}\nonumber
I &=& \int_{\mathbb{R}^n} 2^{j\nu} |\mathcal{F}^{-1}(g_{(j)})(2^{j\vec{a}}y)| \prod\limits_{k=1}^n (1+|2^{ja_k}y_k|)^{1/r_k} dy \nonumber \\
&=& \int_{\mathbb{R}^n} |\mathcal{F}^{-1}(g_{(j)})(x)| \prod\limits_{k=1}^n (1+|x_k|)^{1/r_k} dx,
\end{eqnarray}
where for the last equality we changed to the variable $x:=2^{j\vec{a}}y$.\ \\

We now apply the mixed H\"{o}lder-inequality (\ref{mixedholder}) for the admissible $\vec{t}\in[1,2]^n$ and obtain 
\begin{equation}\label{thproof7}
I \leq \Big\|\mathcal{F}^{-1}(g_{(j)})(x) \prod\limits_{k=1}^n (1+|x_k|)^{N_k}\Big\|_{\vec{t'}}
\Big\|\prod\limits_{k=1}^n (1+|x_k|)^{-\big(\varepsilon+\frac{1}{t_k}\big)}\Big\|_{\vec{t}}.
\end{equation}
Now it is easy to observe that 
$$\Big\|\prod\limits_{k=1}^n (1+|x_k|)^{-\big(\varepsilon+\frac{1}{t_k}\big)}\Big\|_{\vec{t}}=\prod\limits_{k=1}^n \Big\|(1+|x_k|)^{-\big(\varepsilon+\frac{1}{t_k}\big)}\Big\|_{t_k}\leq c.$$
On the other hand, since $$\prod\limits_{k=1}^n (1+|x_k|)^{N_k}\leq (1+|x|)^N\leq c_N \sum\limits_{|\gamma|\leq N} |x^{\gamma}|,$$
we conclude that
\begin{eqnarray}\label{thproof8}\nonumber
I\leq c \Big\|\mathcal{F}^{-1}(g_{(j)})(x) \sum\limits_{|\gamma|\leq N}|x^\gamma|\Big\|_{\vec{t'}}
&\leq& c \sum\limits_{|\gamma|\leq N} \Big\|\mathcal{F}^{-1}(g_{(j)})(x) x^\gamma \Big\|_{\vec{t'}}\nonumber \\\
&=:& c \sum\limits_{|\gamma|\leq N} I_\gamma.
\end{eqnarray}

\noindent\textit{Estimation of} $I_\gamma$.\ \\

We have for every multi-index $|\gamma|\le N$,

\begin{eqnarray*}\nonumber
I_\gamma&=&\Big\|\mathcal{F}^{-1}(g_{(j)})(x)x^\gamma\Big\|_{\vec{t'}}
= \Big\|\mathcal{F}^{-1}(\partial^\gamma g_{(j)})(\cdot) \Big\|_{\vec{t'}}\nonumber \\
&=& \Big\|\mathcal{F}(\partial^\gamma g_{(j)})(-\cdot) \Big\|_{\vec{t'}}
= \Big\|\mathcal{F}(\partial^\gamma g_{(j)})(\cdot) \Big\|_{\vec{t'}}.
\end{eqnarray*}

By assumption, $\vec{t}$ is admissible, so we may apply the mixed Hausdorff-Young inequality (\ref{Penedek}) to  obtain

\begin{equation}\label{thproof9}
I_\gamma=\Big\|\mathcal{F}(\partial^\gamma g_{(j)})(\cdot) \Big\|_{\vec{t'}}
\leq \Big\|\partial^\gamma g_{(j)}(\cdot) \Big\|_{\vec{t}}.
\end{equation}
From the definitions of $m_{(j)}$ and $g_{(j)}$, it follows that $$(\partial^\gamma g_{(j)})(x)=2^{j\vec{a}\gamma}(\partial^\gamma m_{(j)})(2^{j\vec{a}}x).$$ Moreover, by a change of variables, we observe that
$$\|\partial^{\gamma} m_{(j)}(2^{j\vec{a}}\cdot)\|_{\vec{t}}=2^{-j\vec{a}\cdot\frac{1}{\vec{t}}}\|\partial^{\gamma} m_{(j)}(\cdot)\|_{\vec{t}}.$$ Moreover, by Leibniz's product rule,
\begin{eqnarray*}\nonumber
\big|\partial^{\gamma} m_{(j)}(x)\big|&\le& 2^{-j\alpha}\sum\limits_{k=j-M}^{j+M} \big|\partial^{\gamma}(m\widehat{\varphi_k})(x)\big|\nonumber \\
&\le& 2^{-j\alpha}\sum\limits_{k=j-M}^{j+M}\sum\limits_{\beta\leq \gamma} {\textstyle\binom{\gamma}{\beta}} \Big|(\partial^{\beta}m)(x)(\partial^{\gamma-\beta}\widehat{\varphi_k})(x)\Big|\nonumber \\
&\leq& c 2^{-j\alpha}\sum\limits_{\beta\leq \gamma} \big|\partial^{\beta}m(x)\big|,
\end{eqnarray*}
since $\big|(\partial^{\gamma-\beta}\widehat{\varphi_k})(x)\big|=2^{-k\vec{a}\cdot(\gamma-\beta)}\big|(\partial^{\gamma-\beta}\hat{\varphi})(2^{-k\vec{a}}x)\big|\leq c$ as $\beta\leq\gamma$ (recall that $\varphi_k\equiv0$, for $k<0$).

Combining (\ref{thproof9}) with  the above estimates, eq.\ \eqref{mixedRj}, and the fact that $\supp(m_{(j)})\subseteq \bigcup_{k=j-M}^{j+M} R_{k}$, we arrive at
\begin{eqnarray}\label{thproof10}\nonumber
I_\gamma &\leq& c 2^{j\vec{a}\cdot\gamma} 2^{-j\vec{a}\cdot\frac{1}{\vec{t}}}\|\partial^{\gamma} m_{(j)}(\cdot)\|_{L^{\vec{t}}(\cup_{k=j-M}^{j+M} R_{k})}\nonumber \\
&\leq& c 2^{-j\alpha}2^{j\vec{a}\cdot\gamma} 2^{-j\vec{a}\cdot\frac{1}{\vec{t}}}\sum\limits_{\beta\leq\gamma}\|\partial^{\beta} m(\cdot)\|_{L^{\vec{t}}(\cup_{k=j-M}^{j+M} R_{k})}
\nonumber \\
&\leq& c \sum\limits_{\beta\leq\gamma}\sum_{k=j-M}^{j+M}2^{-k\alpha}2^{k\vec{a}\cdot\gamma} 2^{-k\vec{a}\cdot\frac{1}{\vec{t}}}\|\partial^{\beta} m(\cdot)\|_{L^{\vec{t}}(R_{k})}.
\end{eqnarray}
By (\ref{thproof8}) and (\ref{thproof10}) we have
\begin{eqnarray}\label{thproof11}
I&\leq& c \sum\limits_{|\gamma|\leq N}\sum_{k=j-M}^{j+M} 2^{-k\alpha}2^{k\vec{a}\cdot\gamma} 2^{-k\vec{a}\cdot\frac{1}{\vec{t}}}\|\partial^{\beta} m(\cdot)\|_{L^{\vec{t}}(R_{k})}\nonumber \\
&\le& c A_{\alpha,\vec{t},N}(m)< c,
\end{eqnarray}
since $m$ belongs to the family $\mathcal{A}(\alpha,\vec{t},N)$. Combining the last with (\ref{thproof4}) and (\ref{thproof5}) we deduce that
\begin{equation}\label{thproof111}
\big|\big(\varphi_j\ast T_mf \big)(x)\big|\le c2^{j\alpha}\mathcal{M}_{\vec{r}} (\varphi_j\ast f)(x).
\end{equation}

We now pass to the Triebel-Lizorkin norm, and we apply the mixed-Fefferman-Stein maximal inequality  (\ref{max}) to conclude that
\begin{eqnarray}\label{thproof12}
\|T_m f\|_{F^{s}_{\vec{p}q}(\vec{a})}&=&\left\|\Big(\sum\limits_{j=0}^{\infty} (2^{js}|\varphi_j\ast T_m f|)^q\Big)^{1/q}\right\|_{\vec{p}}
\nonumber \\
&\leq& c \left\|\Big(\sum\limits_{j=0}^{\infty} (2^{js}2^{j\alpha}\mathcal{M}_{\vec{r}}(\varphi_j\ast f)^q\Big)^{1/q}\right\|_{\vec{p}}\nonumber \\
&\leq& c \left\|\Big(\sum\limits_{j=0}^{\infty} (2^{j(s+\alpha)}|\varphi_j\ast f|)^q\Big)^{1/q}\right\|_{\vec{p}}\le c\|f\|_{F^{s+\alpha}_{\vec{p}q}(\vec{a})},
\end{eqnarray}
which concludes the proof.

\end{proof}

\begin{remark} Let us consider the case $\vec{t}=(2,\dots,2)$. As  mentioned in Remark \ref{rema1}, our condition coincides with the (anisotropic) H\"{o}rmander condition with the extra parameter $\alpha\in\bR$ (as in \cite{CK,YYZ}). The smoothness level that we require for the multiplier $m$ is $N=\big[\mu+\frac{n}{2}\big]+1$ for $\mu=\frac{1}{\mu_1}+\cdots+\frac{1}{\mu_n}$, where $\mu_j=\min(p_1,\dots,p_j,q)$. This means that we ask for
$$N\le N_0:=\Big[\frac{n}{\min(p_1,\dots,p_n,q)}+\frac{n}{2}\Big]+1$$
derivatives on $m$. 

When $\vec{p}=(p,\dots,p)$, the index $N_0$ becomes the same as the one appearing by  Yang et al. in \cite{YYZ}.

Sharp multiplier results on Triebel-Lizorkin spaces have been proved by Triebel in \cite{T2}. The main tool used by Triebel in \cite{T2} is complex interpolation, and currently such tools are not available in the mixed-norm setting.
\end{remark}

\section{Special Cases}

\subsection{Fourier multipliers on anisotropic mixed-norm Sobolev spaces}

One of the  main motivation for studying Triebel-Lizorkin and Besov spaces is that for specific choices of parameters many well-know spaces of harmonic analysis can be recovered. Let us mention the special case mixed-norm Sobolev and generalized Sobolev spaces, see \cite{L}. Such spaces provide a natural setting for the study of partial differential operators.
 
Let $\vec{p}\in(1,\infty)^n$ and $\vec{k}\in\mathbb{N}_0^n,$ the \textit{mixed-norm Sobolev space} $W^{\vec{k}}_{\vec{p}}$ is defined, as the set of all $f\in \mathcal{S}'$ such that 
\begin{equation}\label{sobnorm}
\|f\|_{W^{\vec{k}}_{\vec{p}}}:=\|f\|_{\vec{p}}+\sum\limits_{j=1}^{n}\Big\|\frac{\partial^{k_j}f}{\partial x_j^{k_j}}\Big\|_{\vec{p}}<\infty.
\end{equation} 
Note that $W^{\vec{0}}_{\vec{p}}=L^{\vec{p}},$ for every $\vec{p}\in(1,\infty)^n.$

Now let $\vec{p}\in(1,\infty)^n,\;s\in\mathbb{R}$ and $\vec{a}\in[1,\infty)^n,$ the \textit{anisotropic mixed-norm generalized Sobolev space} $H^{s}_{\vec{p}}(\vec{a})$ is defined, as the set of all $f\in \mathcal{S}'$ such that 
\begin{equation}\label{gensobnorm}
\|f\|_{H^{s}_{\vec{p}}(\vec{a})}:=\Big\|\cF^{-1}\big((1+|\xi|_{\vec{a}}^2)^{s/2}\hat{f}
\big)(\cdot)\Big\|_{\vec{p}}<\infty.
\end{equation} 

We have the \textit{identifications}:
\

\noindent 1. When $\vec{p}\in(1,\infty)^n,\;s\in\mathbb{R}$ and $\vec{a}\in[1,\infty)^n,$ then
$$F^{s}_{\vec{p}2}(\vec{a})=H^{s}_{\vec{p}}(\vec{a}),\;\;\text{with equivalent norms}.$$
2. When $\vec{p}\in(1,\infty)^n,\;s\in\mathbb{R},\;\vec{k}\in\mathbb{N}_0^n$ and $\vec{a}\in[1,\infty)^n,$ satisfying $k_j=s/a_j,$ for all $j=1,\dots,n,$ then
$$F^{s}_{\vec{p}2}(\vec{a})=W^{\vec{k}}_{\vec{p}},\;\;\text{with equivalent norms. Especially}\;\;F^{0}_{\vec{p}2}(\vec{a})=L^{\vec{p}}.$$

Before we present our Corollaries, we fix an admissible $\vec{t}$ and we keep $q=2$. Then for every $\vec{p}$, we have $\mu_j=\min(p_1,\dots,p_n,2),\;j=1,\dots,n$. Finally recall that $t=\frac{1}{t_1}+\cdots+\frac{1}{t_n}$ and $\mu=\frac{1}{\mu_1}+\cdots+\frac{1}{\mu_n}$.

By all the above, Theorem \ref{Thmain} implies the following:
\begin{corollary}\label{Csobolev} Let $\alpha,s\in\mathbb{R},\;\vec{p}\in(1,\infty)^n,\;N\in\bN$ and $\vec{a}\in[1,\infty)^n.$ If $m\in \mathcal{A}(\alpha,\vec{t},N)$, for $N>\mu+t$, then the Fourier multiplier $T_{m}$ is bounded from $H^{s+\alpha}_{\vec{p}}(\vec{a})$ to $H^{s}_{\vec{p}}(\vec{a}).$
\end{corollary}
By the identification of mixed-norm Triebel-Lizorkin with Sobolev spaces, Theorem \ref{Thmain}  also offers the following Corollary:
\begin{corollary}\label{Csobolev2} Let $\alpha,s\in\mathbb{R},\;\vec{p}\in(1,\infty)^n,\;\vec{a}\in[1,\infty)^n$ be such that
$$\vec{k}:=\left(\frac{s}{a_1},\dots,\frac{s}{a_n}\right),\;\vec{\ell}:=\left(\frac{\alpha}{a_1},\dots,\frac{\alpha}{a_n}\right)\in\mathbb{N}_0^{n}.$$ If $m\in\mathcal{A}(\alpha,\vec{t},N)$, for $N>\mu+t$, then the Fourier multiplier  $T_{m}$ is bounded from $W^{\vec{k}+\vec{\ell}}_{\vec{p}}$ to $W^{\vec{k}}_{\vec{p}}.$
\end{corollary}

Let us restrict our attention  to the isotropic case; $\vec{a}=(1,\dots,1)$. Then we recover the following recent results by Antoni\'{c} and Ivec \cite{AI}:

\begin{corollary}\label{Csobolev3} Let $\alpha,s\in\mathbb{N}_0$ and $\vec{p}\in(1,\infty)^n.$ If $m\in \mathcal{A}(\alpha,\vec{t},N)$, for $N>\mu+t$, then the Fourier multiplier   $T_{m}$ is bounded from $W^{s+\alpha}_{\vec{p}}$ to $W^{s}_{\vec{p}}.$ Especially when $\alpha=0,$ then $T_{m}$ is bounded on $L^{\vec{p}}$.
\end{corollary}

\subsection{Equivalent characterizations} We conclude the paper by considering one explicit example of a bounded Fourier multiplier that can be used to obtain an equivalent norm of anisotropic mixed-norm Besov and Triebel-Lizorkin spaces.

Let $\vec{a}\in [1,\infty)^n$. We consider the following \textit{anisotropic bracket},
$$\brac{x}:=|(1,x)|_{(1,\vec{a})},\qquad x\in\bR^n.$$ 
This quantity has been studied in detail in \cite{BN}. It is known that $\brac{\cdot}\in\cC^\infty(\bR^n)$ and that for $\alpha\in\bR$, $\gamma\in \bN_0^n$, there are constants $c_\gamma, c_\gamma'>0$ such that for every $\xi\in\bR^n$
\begin{equation}\label{dbrac}
 |\partial^\gamma\brac{\xi}^\alpha|\leq c_\gamma\brac{\xi}^{\alpha-\vec{a}\cdot\gamma}\le c_\gamma'(1+|\xi|_{\vec{a}})^{\alpha-\vec{a}\cdot\gamma}. 
  \end{equation}
So the multiplier $m_\alpha (\xi):=\brac{\xi}^\alpha,\;\alpha\in\bR$, satisfies the Mihlin condition (\ref{multiinfty}) for arbitrary $N\in\bN$. By Theorem \ref{Thmain}, the multiplier $T_{m_\alpha}$ is bounded from $F^{s+\alpha}_{\vec{p}q}(\vec{a})$ to $F^{s}_{\vec{p}q}(\vec{a})$
and from $B^{s+\alpha}_{\vec{p}q}(\vec{a})$ to $B^{s}_{\vec{p}q}(\vec{a})$. Moreover, we observe that $T_{m_\alpha}\circ T_{m_{-\alpha}}$ is the identity on $\cS'$ and thus we have the following characterization:

\begin{corollary}\label{Cequivallent} Let $\alpha,s\in\mathbb{R},\;\vec{p}=(p_1,\dots,p_n)\in (0,\infty)^n,\;q\in(0,\infty]$ and $\vec{a}\in [1,\infty)^n$. Then $\|T_{m_\alpha}f\|_{B^{s-\alpha}_{\vec{p}q}(\vec{a})}$ and $\|T_{m_\alpha}f\|_{F^{s-\alpha}_{\vec{p}q}(\vec{a})}$ are equivalent quasi-norms on $B^{s}_{\vec{p}q}(\vec{a})$ and $F^{s}_{\vec{p}q}(\vec{a})$, respectively.
\end{corollary}

\end{document}